\documentclass[12pt]{amsart}

\usepackage[colorlinks=true, linkcolor=blue, citecolor=blue]{hyperref}
\usepackage{amssymb}
\usepackage{amsmath, graphicx, rotating}
\usepackage{color}
\usepackage{soul}

\usepackage[T1]{fontenc}
\usepackage{lmodern}
\usepackage[english,francais]{babel}

\usepackage{ upgreek }
\usepackage{stmaryrd}
\SetSymbolFont{stmry}{bold}{U}{stmry}{m}{n}
\usepackage{amsthm}
\usepackage{float}

\usepackage{ bbm }
\usepackage{ stmaryrd }
\usepackage{ mathrsfs }
\usepackage{ frcursive }
\usepackage{ comment }

\usepackage{pgf, tikz}
\usetikzlibrary{shapes}
\usepackage{varioref}
\usepackage{enumitem}
\usepackage{todonotes}

\definecolor{rouge}{rgb}{0.7,0.00,0.00}
\definecolor{vert}{rgb}{0.00,0.5,0.00}
\definecolor{bleu}{rgb}{0.00,0.00,0.8}
\usepackage[margin=1in]{geometry}
\newtheorem{theorem}{Theorem}[section]
\newtheorem*{theorem*}{Theorem}

\newtheorem{proposition}[theorem]{Proposition}

\labelformat{hypothesis}{\textbf{M\kern-0.1mm#1}}
\newtheorem{hypoP}{P\kern-0.1mm}
\labelformat{hypoP}{\textbf{P\kern-0.1mm#1}}

\newtheorem{lemme}[theorem]{Lemme}

\theoremstyle{definition}

\numberwithin{equation}{section}

\newcommand*{\abs}[1]{\left\lvert#1\right\rvert}
\newcommand*{\norm}[1]{\left\lVert#1\right\rVert}

\newcommand*{\scal}[2]{\left\langle {#1}, {#2} \right\rangle}

\def\bb#1{\mathbb{#1}}
\def\bs#1{\boldsymbol{#1}}
\def\bf#1{\mathbf{#1}}
\def\scr#1{\mathscr{#1}}
\def\bbm#1{\mathbbm{#1}}
\def\geq{\geqslant}
\def\leq{\leqslant}

\def\tt#1{\tilde{#1}}
\def\tbs#1{\tilde{\boldsymbol{#1}}}

\newcommand\ee{\varepsilon}
\renewcommand\ll{\lambda}
\DeclareMathOperator{\GL}{GL}

\DeclareMathOperator{\dd}{d\!}
\DeclareMathOperator{\e}{e}

\DeclareMathOperator{\Var}{Var}
\DeclareMathOperator{\Cov}{Cov}
\DeclareMathOperator{\Id}{Id}
\DeclareMathOperator{\supp}{supp}

\def\l({\left(}
\def\r){\right)}

\begin{document}

\title[Espace de Banach pour le produit de matrices]{Construction d'un espace de Banach pour le produit de matrices aléatoires}
\author{Ion Grama, Ronan Lauvergnat, Emile Le Page}

\date{\today}

\maketitle

Le but de cet article est de montrer que les Théorèmes 2.2-2.5 de l'article \cite{grama_limit_2018-1}
s'appliquent au produit de matrices aléatoires considéré par Grama, Le Page et Peigné \cite{grama_conditioned_2016}. 
Cela nous permet en particulier d'insister sur l'aspect général que revêt la formulation de nos théorèmes en \cite{grama_limit_2018-1} 
en montrant que nos hypothèses sont vérifiées pour les modèles antérieurs.

Il est important de souligner que l'espace de Banach qui est construit dans \cite{grama_conditioned_2016} n'est pas entièrement adapté à nos Hypothèses M1-M5. 
En effet le point délicat est le fait que dans l'Hypothèse M4 on suppose que la fonction $f$ est bornée par une fonction de l'espace de Banach. Or dans l'article \cite{grama_conditioned_2016}, l'espace de Banach est inclus dans les fonctions bornées tandis que la fonction $f=\rho$ n'est pas bornée. Il nous faut donc modifier l'espace donné dans \cite{grama_conditioned_2016} et le remplacer par un nouvel espace de Banach que l'on introduit ci-dessous. Commençons par réintroduire des notations de \cite{grama_conditioned_2016}.

\subsection{Rappel des notations}

Soit $\bb G = \GL_d(\bb R)$ l'ensemble des matrices inversibles de taille $d\times d$, $d\geq 1$. On munit $\bb R^d$ de sa norme euclidienne, $\norm{v}= \sqrt{\sum_{i=1}^d v_i^1}$, pour tout $v=(v_1,\dots,v_d) \in \bb R^d$ et $\bb G$ de la norme d'opérateur associée, $\norm{g} = \sup_{v \in \bb R^d \setminus \{0\}} \norm{gv}/\norm{v}$, pour tout $g \in \bb G$. On note également $\bb P (\bb R^d)$ l'espace projectif associé et pour tout $v \in \bb R^d$, on désignera par $\overline{v} \in \bb P(\bb R^d)$ sa direction. On munit $\bb P (\bb R^d)$ de la distance angulaire $d(\overline{u},\overline{v}) = \norm{u \wedge v}/(\norm{u} \norm{v})$, où $u\wedge v$ est le produit vectoriel de $u$ et $v$.
 Le groupe $\bb G$ agit sur l'espace projectif $\bb P(\bb R^d)$ par multiplication : pour tout $\overline{v} \in \bb R^d$, on note $g \cdot \overline{v} = \overline{gv}$ l'action de $g$ sur la direction $\overline{v}$. Enfin pour tout $g \in \bb G$, on pose
\[
N(g) = \max \l( \norm{g}, \norm{g^{-1}} \r).
\]

Soit maintenant $(\Omega, \scr F, \bb P)$ un espace probabilisé, $\bb E$ l'espérance associée et $( g_n )_{n\geq 1}$ une suite de variables aléatoires \textit{i.i.d.}\ définie sur $\Omega$ et à valeurs dans $\bb G$ dont la loi commune est notée $\bs \mu$. Rappelons les hypothèses de \cite{grama_conditioned_2016}.
\begin{hypoP}
\label{HypoP1}
Il existe $\delta_0 >0$ tel que 
\[
\bb E \l( N(g_1)^{\delta_0} \r) = \int_{\bb G} \exp\l( \delta_0 \log \l( N(g) \r) \r) \bs \mu(g) < +\infty.
\]
\end{hypoP}
\begin{hypoP}[\textbf{Irréductibilité forte}]
\label{HypoP2}
L'action du support de $\bs \mu$ sur $\bb R^d$ est fortement irréductible c'est-à-dire qu'il n'existe pas d'union propre de sous-espace de $\bb R^d$ qui soit invariante par $\Gamma_{\mu}$, le plus petit semi-groupe fermé contenant le support de $\bs \mu$.
\end{hypoP}
\begin{hypoP}[\textbf{Propriété de contraction}]
\label{HypoP3}
Le semi-groupe $\Gamma_{\mu}$ contient une suite contractante.
\end{hypoP}
Soit $\rho$ le cocycle défini par
\[
\rho(g,\overline{v}) := \log \l( \frac{\norm{g v}}{\norm{v}} \r), \qquad \forall (g,\overline{v}) \in \bb G \times \bb P(\bb R^d).
\]
Sous les conditions \ref{HypoP1}-\ref{HypoP3}, on sait qu'il existe une unique mesure $\bs \nu$ qui soit $\bs \mu$-invariante sur $\bb P(\bb R^d)$.
\begin{hypoP}
\label{HypoP4}
Le plus grand exposant de Lyapunov vaut $0$ : $\int_{\bb G \times \bb P(\bb R^d)} \rho(g,\overline{v}) \bs \mu(\dd g) \bs \nu(\dd \overline{v}) = 0$.
\end{hypoP}
La condition $\bf{P5}$ de Grama, Le Page et Peigné \cite{grama_conditioned_2016} assure que la fonction harmonique est positive pour tout $y > 0$. Cette hypothèse n'est pas nécessaire dans notre propos et la Proposition \ref{Macédoine1} explicitera le domaine exact de positivité de cette fonction harmonique. 
%
Pour des explications sur ces conditions, on renvoie à l'article \cite{grama_conditioned_2016}. Introduisant maintenant la marche aléatoire associée à ce produit de matrices. Pour tout $n \geq 1$, on définit le produit de matrice
\[
G_n := g_n \dots g_1 \qquad \text{et} \qquad G_0 = \Id.
\]
Soit $\bb B$ la boule unité fermée de $\bb R^d$. Pour étudier le premier instant pour lequel le produit $G_n v$, pour $v \notin \bb B$ entre dans la boule unité $\bb B$, on considère le logarithme de sa norme
\[
\log\l( \norm{G_n v} \r) = \sum_{k=1}^{n} \rho \l( g_k, G_{k-1} \cdot \overline{v} \r) + \log(\norm{v}).
\]
Soit $\bb X = \bb G \times \bb P(\bb R^d)$. Pour $x=(g,\overline{v}) \in \bb X$, on considère $( X_n )_{n\geq 0}$ la chaîne de Markov sur $\Omega$ à valeurs dans $\bb X$ définie par $X_0 = x$ et
\[
X_n  = \l( g_n, G_{n-1}g\cdot \overline{v} \r), \qquad \forall n \geq 1.
\]
La marche markovienne associée est alors donnée par $S_n = \rho(X_1) + \dots \rho( X_n )$.

Ces rappels étant posés on introduit désormais l'espace de Banach que nous considérerons.

\subsection{L'espace de Banach}

On note $\scr C \l( \bb X, \bb C \r)$ l'ensemble des fonctions continues sur $\bb X$ et à valeurs dans $\bb C$. On fixe les paramètres suivants
\[
\ee = \frac{\delta_0}{8}, \qquad \theta = 3\ee = \frac{3\delta_0}{8}, \qquad \alpha = 5\ee = \frac{5\delta_0}{8}, \qquad \beta = 7\ee = \frac{7\delta_0}{8},
\]
où $\delta_0$ est défini par \ref{HypoP1}.
Pour toute fonction $h \in \scr C \l( \bb X, \bb C \r)$, on pose
\begin{align*}
\abs{h}_{\theta} &= \sup_{(g,\overline{u}) \in \bb X} \frac{\abs{h(g,\overline{u})}}{N(g)^{\theta}}, \\
k_{\ee,\alpha}(h) = \sup_{\substack{g \in \bb G \\ \overline{u} \neq \overline{v}}} \frac{\abs{h(g,\overline{u})-h(g,\overline{v})}}{d(\overline{u},\overline{v})^{\ee} N(g)^{\alpha}},
&\qquad k_{\ee,\beta}'(h) = \sup_{\substack{g \neq g' \\ \overline{u} \in \bb P(\bb R^d)}} \frac{\abs{h(g,\overline{u})-h(g',\overline{u})}}{\norm{g-g'}^{\ee} N(g)^{\beta} N(g')^{\beta}}.
\end{align*}
On définit alors la norme
\[
\norm{h}_{\scr B} := \abs{h}_{\theta} + k_{\ee,\alpha}(h) + k_{\ee,\beta}'(h),
\]
et l'espace de Banach associé
\[
\scr B := \left\{ h \in \scr C \l( \bb X, \bb C \r) : \norm{h}_{\scr B} < +\infty \right\}.
\]

\subsection{Preuve de M1}

Afin de faciliter la lecture on redonne au fur et à mesure les Hypothèses M1-M5

\begin{proposition}[Espace de Banach]
\label{Canada1}
Supposons \ref{HypoP1}. Alors
\begin{enumerate}[ref=\arabic*, leftmargin=*, label=\arabic*.]
	\item La fonction constante égale à $1$ notée $e$ appartient à $\scr B$.
	\item Pour tout $x\in \bb X$, la mesure de Dirac $\bs \delta_x$ appartient au dual de $\scr B$ noté $\scr B'$.
	\item L'espace de Banach $\scr B$ est inclus dans $L^1\l( \bf P (x,\cdot) \r)$, pour tout $x\in \bb X$.
	\item Il existe une constante $\kappa \in (0,1)$ telle que pour tout $h\in \scr B$, la fonction $\e^{it\rho}h$ est dans $\scr B$ pour tout $t$ vérifiant $\abs{t} \leq \kappa$.
\end{enumerate}
\end{proposition}

\begin{proof} \textit{Point 1.} Puisque $N(g) \geq 1$ pour tout $g \in \bb G$, il est clair que $e \in \scr B$.

\textit{Point 2.} Pour tout $x = (g,\overline{u}) \in \bb X$ et $h \in \scr B$,
\[
\abs{\bs \delta_x(h)} = \abs{h(x)} \leq N(g)^{\theta} \abs{h}_{\theta} \leq N(g)^{\theta} \norm{h}_{\theta}.
\]
Donc $\bs \delta_x \in \scr B'$ et 
\begin{equation}
\label{Bulgarie}
\norm{\bs \delta_x}_{\scr B'} \leq N(g)^{\theta}.
\end{equation}

\textit{Point 3.} Pour tout $x = (g,\overline{u}) \in \bb X$ et $h \in \scr B$,
\[
\bf P \abs{h}(x) = \int_{\bb G} \abs{h(g_1,g\cdot\overline{u})} \bs \mu(\dd g_1) \leq \abs{h}_{\theta} \int_{\bb G} N(g_1)^{\theta} \bs \mu(\dd g_1).
\]
Puisque $\theta = 3\delta_0/8 \leq \delta_0$, par \ref{HypoP1}, on a
\[
\bf P \abs{h}(x) \leq c \abs{h}_{\theta} <+\infty.
\]

\textit{Point 4.} Soit $h \in \scr B$ et $t \in \bb R$. On note tout d'abord que 
\begin{equation}
\label{Etats-Unis001}
\abs{\e^{it\rho}h}_{\theta} = \abs{h}_{\theta}.
\end{equation}
Ensuite pour tout $(g,\overline{u})$ et $(g,\overline{v})$ dans $\bb X$, on écrit que
\begin{equation}
\label{Mexique}
\abs{\e^{it\rho(g,\overline{u})}h(g,\overline{u}) - \e^{it\rho(g,\overline{v})}h(g,\overline{v})} \leq \abs{h}_{\theta} N(g)^{\theta} \abs{\e^{it\rho(g,\overline{u})} - \e^{it\rho(g,\overline{v})}} + \abs{h(g,\overline{u}) - h(g,\overline{v})}.
\end{equation}
Soient $u$ et $v$ deux représentants de $\overline{u}$ respectivement $\overline{v}$, de norme $1$ et de même sens ($\scal{u}{v} \geq 0$). On a
\[
\abs{\e^{it\rho(g,\overline{u})} - \e^{it\rho(g,\overline{v})}} \leq \abs{t} \abs{ \rho(g,\overline{u}) - \rho(g,\overline{v}) } = \abs{t} \abs{\log \l( \frac{\norm{gu}}{\norm{gv}} \r)}.
\]
En utilisant le fait que $\abs{\log(s)} \leq \abs{1-s}$ pour tout $s > 0$,
\[
\abs{ \rho(g,\overline{u}) - \rho(g,\overline{v}) } \leq \frac{\norm{g(u-v)}}{\norm{gv}} \leq \norm{g}\norm{g^{-1}} \norm{u-v}.
\]
Rappelons que
\begin{equation}
\label{Norvège}
\norm{u-v} \leq \sqrt{2} d(\overline{u},\overline{v}) \qquad \text{et que} \qquad d(\overline{u},\overline{v}) \leq \norm{u-v}.
\end{equation}
Donc
\begin{equation}
\label{Islande}
\abs{ \rho(g,\overline{u}) - \rho(g,\overline{v}) } \leq \sqrt{2} N(g)^2 d(\overline{u},\overline{v}).
\end{equation}
D'où,
\[
\abs{\e^{it\rho(g,\overline{u})} - \e^{it\rho(g,\overline{v})}} \leq \sqrt{2} \abs{t} N(g)^2 d(\overline{u},\overline{v}).
\]
D'autre part $\abs{\e^{it\rho(g,\overline{u})} - \e^{it\rho(g,\overline{v})}} \leq 2$, donc
\[
\abs{\e^{it\rho(g,\overline{u})} - \e^{it\rho(g,\overline{v})}} \leq 2^{1-\ee+\ee/2} \abs{t}^{\ee} N(g)^{2\ee} d(\overline{u},\overline{v})^{\ee}.
\]
De \eqref{Mexique}, on en déduit que
\[
\abs{\e^{it\rho(g,\overline{u})}h(g,\overline{u}) - \e^{it\rho(g,\overline{v})}h(g,\overline{v})} \leq 2 \abs{t}^{\ee} \abs{h}_{\theta} N(g)^{\theta+2\ee} d(\overline{u},\overline{v})^{\ee} + k_{\ee,\alpha}(h) d(\overline{u},\overline{v})^{\ee} N(g)^{\alpha}.
\]
Puisque $\alpha = \theta +2\ee$, on a
\begin{equation}
\label{Etats-Unis002}
k_{\ee,\alpha}\l( \e^{it\rho}h \r) \leq 2 \abs{t}^{\ee} \abs{h}_{\theta} + k_{\ee,\alpha}(h) <+\infty.
\end{equation}

On procède de même pour $k_{\ee,\beta}'\l( \e^{it\rho}h \r)$. Soient $(g,\overline{u})$ et $(g',\overline{u})$ dans $\bb X$. Alors,
\[
\abs{\e^{it\rho(g,\overline{u})}h(g,\overline{u}) - \e^{it\rho(g',\overline{u})}h(g',\overline{u})} \leq \abs{h}_{\theta} N(g)^{\theta} \abs{\e^{it\rho(g,\overline{u})} - \e^{it\rho(g',\overline{u})}} + \abs{h(g,\overline{u}) - h(g',\overline{u})}.
\]
Comme précédemment,
\[
\abs{\e^{it\rho(g,\overline{u})} - \e^{it\rho(g',\overline{u})}} \leq \abs{t} \frac{\norm{(g'-g)u}}{\norm{g'u}} \leq \abs{t} N(g') \norm{g-g'}.
\]
Puisque $\abs{\e^{it\rho(g,\overline{u})} - \e^{it\rho(g',\overline{u})}} \leq 2$, on en déduit que
\[
\abs{\e^{it\rho(g,\overline{u})} - \e^{it\rho(g',\overline{u})}} \leq 2 \abs{t}^{\ee} N(g')^{\ee} \norm{g-g'}^{\ee}.
\]
D'où,
\begin{align*}
\abs{\e^{it\rho(g,\overline{u})}h(g,\overline{u}) - \e^{it\rho(g',\overline{u})}h(g',\overline{u})} &\leq 2 \abs{t}^{\ee} \abs{h}_{\theta} N(g)^{\theta}N(g')^{\ee} \norm{g-g'}^{\ee} \\
&\hspace{3cm}+ k_{\ee,\beta}'(h) \norm{g-g'}^{\ee} N(g)^{\beta}N(g')^{\beta}.
\end{align*}
Comme $\ee \leq \beta$ et $\theta \leq \beta$,
\begin{equation}
\label{Etats-Unis003}
k_{\ee,\beta}'\l( \e^{it\rho}h \r) \leq 2 \abs{t}^{\ee} \abs{h}_{\theta} + k_{\ee,\beta}'(h) <+\infty.
\end{equation}
En utilisant \eqref{Etats-Unis001}, \eqref{Etats-Unis002} et \eqref{Etats-Unis003}, on conclut que pour tout $t \in \bb R$, $\e^{it\rho}h \in \scr B$. De plus,
\[
\norm{\e^{it\rho}h}_{\scr B} \leq \norm{h}_{\scr B} + 4 \abs{t}^{\ee} \abs{h}_{\theta}.
\]
\end{proof}

\subsection{Preuve de M2 et M3}

On rappelle que l'opérateur perturbé est donné par $\bf P_t h(x) = \bf P ( \e^{it\rho} h )(x)$ pour tout $t\in \bb R$, $h \in \scr B$ et $x \in \bb X$. L'objectif est alors de montrer que l'opérateur perturbé $\bf P_t$, $t \in \bb R$ vérifie les hypothèses du théorème de Ionescu-Tulcea et Marinescu \cite{tulcea_theorie_1950}. Ceci impliquera notamment que $\bf P$ possède un trou spectral et nous pourrons alors établir la Proposition \ref{Canada2}. Auparavant, on rappelle un résultat de Le Page \cite{page_theoremes_1982} (Théorème 1). On trouvera également un énoncé de cette proposition dans Grama, Le Page et Peigné \cite{grama_conditioned_2016} (Proposition 8.6).

\begin{proposition}
\label{Finlande}
Sous les conditions \ref{HypoP1}-\ref{HypoP3}, il existe $\ee_0 > 0$ et $r_{\ee_0} \in(0,1)$ tels que
\[
\lim_{n\to+\infty} \sup_{\overline{u} \neq \overline{v}} \bb E \l( \frac{ d\l( G_n \cdot \overline{u}, G_n \cdot \overline{v} \r)^{\ee_0} }{ d\l( \overline{u}, \overline{v} \r)^{\ee_0} } \r)^{1/n} = r_{\ee_0}.
\]
\end{proposition}

On fournit dans le lemme suivant un contrôle de la norme de $\bf P_t$.
\begin{lemme}
\label{Biélorussie}
Supposons \ref{HypoP1}-\ref{HypoP3}.
Pour tout $t \in \bb R$, $n \geq 1$ et $h \in \scr B$, la fonction $\bf P_t^n h$ est un élément de $\scr B$ et de plus,
\[
\norm{\bf P_t^n h}_{\scr B} \leq c_{\ee} \l( 1+ \abs{t}^{\ee} \r) \abs{h}_{\theta} + c_{\ee} k_{\ee,\alpha}(h) r_{\ee}^n.
\]
\end{lemme}

\begin{proof}
Fixons $t \in \bb R$ et $n \geq 1$. Pour tout $h \in \scr B$, $x =(g,\overline{u}) \in \bb X$ et $n \geq 1$ la représentation habituelle de l'opérateur $\bf P_t^n$ est donnée par
\[
\bf P_t^nh(x) = \bb E_x \l( \e^{itS_n} h(X_n) \r).
\]
Puisque $S_n \in \bb R$, par définition de $X_n$,
\begin{equation}
\label{Lettonie}
\abs{\bf P_t^nh(x)} \leq \bb E_x \l( \abs{ h(g_n,G_{n-1}g\cdot \overline{u})} \r) \leq \abs{h}_{\theta} \bb E \l( N(g_n)^{\theta} \r)
\end{equation}
Donc comme $\theta = 3\delta_0/8 \leq \delta_0$, par \ref{HypoP1},
\begin{equation}
\label{Estonie1}
\abs{\bf P_t^nh}_{\theta} \leq \abs{h}_{\theta} \bb E \l( N(g_n)^{\theta} \r) < +\infty.
\end{equation}

Pour tout $x \in \bb X$, on note $X_n^x$ la chaîne de Markov issue de $X_0 = x$ et $S_n^x$ la marche associée. Pour tout $(g,\overline{u})$ et $(g,\overline{v})$ éléments de $\bb X$,
\begin{align*}
\Delta_n := \abs{\bf P_t^nh(g,\overline{u}) - \bf P_t^nh(g,\overline{v})} &= \abs{\bb E \l( \e^{it S_n^{(g,\overline{u})}} h\l( X_n^{(g,\overline{u})} \r) - \e^{it S_n^{(g,\overline{v})}} h\l( X_n^{(g,\overline{v})} \r) \r)} \\
&\leq \bb E \l( \abs{h\l( g_n,G_{n-1}g \cdot \overline{u} \r) - h\l( g_n,G_{n-1}g \cdot \overline{v} \r)} \r) \\
&\qquad+ \bb E \l( \abs{h\l( g_n,G_{n-1}g \cdot \overline{u} \r)} \abs{\e^{it S_n^{(g,\overline{u})}}  - \e^{it S_n^{(g,\overline{v})}}} \r) \\
&\leq k_{\ee,\alpha}(h) \bb E \l( d\l( G_{n-1}g \cdot \overline{u}, G_{n-1}g \cdot \overline{v} \r)^{\ee} N(g_n)^{\alpha} \r) \\
&\qquad+ \abs{h}_{\theta} \bb E \l( N(g_n)^{\theta} \abs{\e^{it S_n^{(g,\overline{u})}}  - \e^{it S_n^{(g,\overline{v})}}} \r)
\end{align*}
Or,
\[
\abs{\e^{it S_n^{(g,\overline{u})}}  - \e^{it S_n^{(g,\overline{v})}}} \leq \min \l( \abs{t} \abs{ S_n^{(g,\overline{u})} - S_n^{(g,\overline{v})} }, 2 \r) \leq 2^{1-\ee} \abs{t}^{\ee} \abs{ S_n^{(g,\overline{u})} - S_n^{(g,\overline{v})} }^{\ee}.
\]
Donc en utilisant l'indépendance des $g_i$, $i \geq 1$,
\begin{align*}
\Delta_n &\leq k_{\ee,\alpha}(h) \bb E \l( d\l( G_{n-1}g \cdot \overline{u}, G_{n-1}g \cdot \overline{v} \r)^{\ee} \r) \bb E \l( N(g_n)^{\alpha} \r) \\
&\qquad+ 2^{1-\ee} \abs{t}^{\ee} \abs{h}_{\theta} \bb E \l( N(g_n)^{\theta} \abs{\rho\l( g_n,G_{n-1}g \cdot \overline{u} \r) - \rho\l( g_n,G_{n-1}g \cdot \overline{v} \r)}^{\ee} \r) \\
&\qquad+ 2^{1-\ee} \abs{t}^{\ee} \abs{h}_{\theta} \sum_{k=1}^{n-1} \bb E\l( N(g_n)^{\theta}  \r) \bb E \l( \abs{\rho\l( g_k,G_{k-1}g \cdot \overline{u} \r) - \rho\l( g_k,G_{k-1}g \cdot \overline{v} \r)}^{\ee} \r).
\end{align*}
En utilisant \eqref{Islande} et le fait que $\theta+2\ee = 5\ee = \alpha$,
\begin{align}
\Delta_n &\leq \l( k_{\ee,\alpha}(h) + 2^{1-\ee/2} \abs{t}^{\ee} \abs{h}_{\theta} \r) \bb E \l( d\l( G_{n-1}g \cdot \overline{u}, G_{n-1}g \cdot \overline{v} \r)^{\ee} \r) \bb E \l( N(g_1)^{\alpha} \r) \nonumber\\
&\qquad+ 2^{1-\ee/2} \abs{t}^{\ee} \abs{h}_{\theta} \bb E\l( N(g_1)^{\theta}  \r) \sum_{k=1}^{n-1}  \bb E \l( N(g_1)^{2\ee} \r) \bb E \l( d\l( G_{k-1}g \cdot \overline{u}, G_{k-1}g \cdot \overline{v} \r)^{\ee} \r).
\label{Suède}
\end{align}
Invoquons à présent le fait que la suite $( G_k )_{k \geq 0}$ contracte les directions. Sans perte de généralité, on peut supposer que $\delta_0/8 \leq \ee_0$. Donc par la Proposition \ref{Finlande}, il existe $n_0$ et $r_{\ee} \in (0,1)$ tels que pour tout $n \geq n_0$ et $(\overline{u}, \overline{v}) \in \bb P (\bb R^d)^2$,
\[
\bb E \l( d\l( G_ng \cdot \overline{u}, G_ng \cdot \overline{v} \r)^{\ee} \r) \leq r_{\ee}^n d\l( g\cdot\overline{u}, g\cdot\overline{v} \r)^{\ee}.
\]
Par \eqref{Norvège}, avec $u$ et $v$ deux représentants de $\overline{u}$ respectivement $\overline{v}$, de norme $1$ et de même sens ($\scal{u}{v} \geq 0$),
\begin{align*}
d\l( g\cdot\overline{u}, g\cdot\overline{v} \r) \leq \norm{\frac{gu}{\norm{gu}} - \frac{gv}{\norm{gv}}} &\leq \frac{\norm{g(u-v)}}{\norm{gu}} + \norm{gv} \abs{\frac{1}{\norm{gu}} - \frac{1}{\norm{gv}}} \\
&\leq N(g)^2 \norm{u-v} + \frac{\abs{\norm{gv} - \norm{gu}}}{\norm{gu}} \\
&\leq 2N(g)^2 \norm{u-v} \\
&\leq 2\sqrt{2} N(g)^2 d\l( \overline{u}, \overline{v} \r).
\end{align*}
Donc pour tout $n \geq n_0$
\[
\bb E \l( d\l( G_ng \cdot \overline{u}, G_ng \cdot \overline{v} \r)^{\ee} \r) \leq 2^{3\ee/2} r_{\ee}^n N(g)^{2\ee} d\l( \overline{u}, \overline{v} \r)^{\ee},
\]
De même, pour tout $n \leq n_0$,
\begin{align*}
\bb E \l( d\l( G_ng \cdot \overline{u}, G_ng \cdot \overline{v} \r)^{\ee} \r) \leq \bb E \l( N(G_n)^{2\ee} \r) N(g)^{2\ee} d\l( \overline{u}, \overline{v} \r)^{\ee} &\leq \bb E \l( N(g_1)^{2\ee} \r)^n N(g)^{2\ee} d\l( \overline{u}, \overline{v} \r)^{\ee} \\
&\leq c_{\ee} N(g)^{2\ee} d\l( \overline{u}, \overline{v} \r)^{\ee},
\end{align*}
où $c_{\ee}$ est une constante numérique ne dépendant que de $\ee$ et dont la valeur est susceptible de changer dans la suite. On obtient que pour tout $n \geq 1$,
\[
\bb E \l( d\l( G_ng \cdot \overline{u}, G_ng \cdot \overline{v} \r)^{\ee} \r) \leq c_{\ee} r_{\ee}^n N(g)^{2\ee} d\l( \overline{u}, \overline{v} \r)^{\ee}.
\]
Ainsi, en injectant cette inégalité dans \eqref{Suède}
\[
\Delta_n \leq c_{\ee} k_{\ee,\alpha}(h) r_{\ee}^{n-1} N(g)^{2\ee} d\l( \overline{u}, \overline{v} \r)^{\ee} + c_{\ee} \abs{t}^{\ee} \abs{h}_{\theta} N(g)^{2\ee} d\l( \overline{u}, \overline{v} \r)^{\ee} \sum_{k=1}^{n} r_{\ee}^{k-1}.
\]
Puisque $2\ee \leq \alpha = 5\ee$, on en déduit que
\begin{equation}
\label{Estonie2}
k_{\ee,\alpha}(\bf P_t^n h) \leq c_{\ee} \abs{t}^{\ee} \abs{h}_{\theta} + c_{\ee} k_{\ee,\alpha}(h) r_{\ee}^n < +\infty.
\end{equation}

De même pour tout $(g,\overline{u})$ et $(g',\overline{u})$ éléments de $\bb X$,
\begin{align*}
\Delta_n' &:= \abs{\bf P_t^nh(g,\overline{u}) - \bf P_t^nh(g',\overline{u})} \\
 &\leq k_{\ee,\alpha}(h) \bb E \l( d\l( G_{n-1}g \cdot \overline{u}, G_{n-1}g' \cdot \overline{u} \r)^{\ee} \r) \bb E \l( N(g_n)^{\alpha} \r) \\
&\qquad+ 2^{1-\ee} \abs{t}^{\ee} \abs{h}_{\theta} \bb E \l( N(g_n)^{\theta} \abs{\rho\l( g_n,G_{n-1}g \cdot \overline{u} \r) - \rho\l( g_n,G_{n-1}g' \cdot \overline{u} \r)}^{\ee} \r) \\
&\qquad+ 2^{1-\ee} \abs{t}^{\ee} \abs{h}_{\theta} \sum_{k=1}^{n-1} \bb E\l( N(g_n)^{\theta}  \r) \bb E \l( \abs{\rho\l( g_k,G_{k-1}g \cdot \overline{u} \r) - \rho\l( g_k,G_{k-1}g' \cdot \overline{u} \r)}^{\ee} \r).
\end{align*}
En utilisant à nouveau \eqref{Islande},
\[
\Delta_n' \leq c_{\ee} k_{\ee,\alpha}(h) \bb E \l( d\l( G_{n-1}g \cdot \overline{u}, G_{n-1}g' \cdot \overline{u} \r)^{\ee} \r) + c_{\ee} \abs{t}^{\ee} \abs{h}_{\theta} \sum_{k=1}^{n} \bb E \l( d\l( G_{k-1}g \cdot \overline{u}, G_{k-1}g' \cdot \overline{u} \r)^{\ee} \r).
\]
Or comme précédemment, par \eqref{Norvège},
\[
d\l( g\cdot\overline{u}, g'\cdot\overline{u} \r) \leq \norm{\frac{gu}{\norm{gu}} - \frac{g'u}{\norm{g'u}}} \leq \frac{\norm{(g-g')u}}{\norm{gu}} + \norm{g'u} \abs{\frac{1}{\norm{gu}} - \frac{1}{\norm{g'u}}} \leq 2N(g) \norm{g-g'}.
\]
Donc par la Proposition \ref{Finlande}, pour tout $n \geq 1$,
\[
\bb E \l( d\l( G_ng \cdot \overline{u}, G_ng \cdot \overline{v} \r)^{\ee} \r) \leq c_{\ee} r_{\ee}^n N(g)^{\ee} \norm{g-g'}^{\ee}.
\]
Ainsi,
\[
\Delta_n' \leq c_{\ee} k_{\ee,\alpha}(h) r_{\ee}^n N(g)^{\ee} \norm{g-g'}^{\ee} + c_{\ee} \abs{t}^{\ee} \abs{h}_{\theta} N(g)^{\ee} \norm{g-g'}^{\ee}.
\]
Puisque $\beta \geq \ee$ et que $N(g') \geq 1$, on obtient que
\begin{equation}
\label{Estonie3}
k_{\ee,\beta}'(\bf P_t^n h) \leq c_{\ee} k_{\ee,\alpha}(h) r_{\ee}^n + c_{\ee} \abs{t}^{\ee} \abs{h}_{\theta} < +\infty.
\end{equation}
En ajoutant \eqref{Estonie1}, \eqref{Estonie2} et \eqref{Estonie3}, on achève la démonstration.
\end{proof}

On peut maintenant montrer que les conditions du théorème de Ionescu-Tulcea et Marinescu \cite{tulcea_theorie_1950} sont vérifiés sous \ref{HypoP1}-\ref{HypoP3}. On pourra se référer au livre de Norman \cite{norman1972markov}.

\begin{lemme}
\label{Ukraine}
Supposons \ref{HypoP1}-\ref{HypoP3}.
L'espace de Banach $(\scr B, \norm{\cdot}_{\scr B})$ est inclus dans $( \scr C(\bb X, \bb C), \abs{\cdot}_{\theta})$ qui est lui aussi un espace de Banach.
\begin{enumerate}[ref=\arabic*, leftmargin=*, label= \arabic*.]
\item Soient $( h_n )_{n\geq 1} \in \scr B^{\bb N}$ et $h \in \scr C(\bb X,\bb C)$ tels que $\abs{h_n-h}_{\theta} \to 0$ quand $n \to +\infty$ et tels que pour tout $n \geq 1$, $\norm{h_n}_{\scr B} \leq C$. Alors $h \in \scr B$ et de plus $\norm{h}_{\scr B} \leq C$.
\item On a pour tout $t \in \bb R$ et $h \in \scr B$,  
\[
\sup_{n\geq 0} \frac{\abs{\bf P_t^n h}_{\theta}}{\abs{h}_{\theta}} < +\infty.
\]
\item Pour tout $t \in \bb R$, il existe $k_0 \geq 1$, $r \in (0,1)$ et $c >0$ tels que pour tout $k \geq k_0$ et tout $h \in \scr B$,
\[
\norm{\bf P_t^k h}_{\scr B} \leq r \norm{h}_{\scr B} + c \abs{h}_{\theta}.
\]
\item Pour tout $t \in \bb R$, l'opérateur $\bf P_t$ de $(\scr B, \norm{\cdot})$ à $(\scr C(\bb X,\bb C), \abs{\cdot}_{\theta})$ est compact : pour toute partie bornée $B$ de $\scr B$, l'ensemble $\bf P_t B$ est relativement compact.
\end{enumerate}
\end{lemme}

\begin{proof}
\textit{Point 1.} Soient $( h_n )_{n\geq 1} \in \scr B^{\bb N}$ et $h \in \scr C(\bb X,\bb C)$. On suppose que $\abs{h_n-h}_{\theta} \to 0$ quand $n \to +\infty$ et que pour tout $n \geq 1$, $\norm{h_n}_{\scr B} \leq C$. Pour tout $(g_1,g_2,g_3,g_3') \in \bb G^4$, $(\overline{u_1},\overline{u_2},\overline{u_2}',\overline{u_3}) \in \bb P(\bb R^d)^4$ et tout $n \geq 0$, on écrit que
\begin{align*}
&\frac{\abs{h(g_1,\overline{u_1})}}{N(g_1)^{\theta}} + \frac{\abs{h(g_2,\overline{u_2})-h(g_2,\overline{u_2}')}}{d(\overline{u_2},\overline{u_2}')^{\ee} N(g_2)^{\alpha}} + \frac{\abs{h(g_3,\overline{u_3})-h(g_3',\overline{u_3})}}{\norm{g_3-g_3'}^{\ee} N(g_3)^{\beta} N(g_3')^{\beta}} \\
&\hspace{2cm}\leq \frac{\abs{h_n(g_1,\overline{u_1})}}{N(g_1)^{\theta}} + \frac{\abs{h_n(g_2,\overline{u_2})-h_n(g_2,\overline{u_2}')}}{d(\overline{u_2},\overline{u_2}')^{\ee} N(g_2)^{\alpha}} + \frac{\abs{h_n(g_3,\overline{u_3})-h_n(g_3',\overline{u_3})}}{\norm{g_3-g_3'}^{\ee} N(g_3)^{\beta} N(g_3')^{\beta}} \\
&\hspace{4cm} + \abs{h-h_n}_{\theta} \l( 1+ \frac{2 N(g_2)^{\theta}}{d(\overline{u_2},\overline{u_2}')^{\ee} N(g_2)^{\alpha}} + \frac{N(g_3)^{\theta} + N(g_3')^{\theta}}{\norm{g_3-g_3'}^{\ee} N(g_3)^{\beta} N(g_3')^{\beta}} \r)
\end{align*}
Donc
\begin{align*}
&\frac{\abs{h(g_1,\overline{u_1})}}{N(g_1)^{\theta}} + \frac{\abs{h(g_2,\overline{u_2})-h(g_2,\overline{u_2}')}}{d(\overline{u_2},\overline{u_2}')^{\ee} N(g_2)^{\alpha}} + \frac{\abs{h(g_3,\overline{u_3})-h(g_3',\overline{u_3})}}{\norm{g_3-g_3'}^{\ee} N(g_3)^{\beta} N(g_3')^{\beta}} \\
&\hspace{2cm} \leq C + \abs{h-h_n}_{\theta} \l( 1+ \frac{2 N(g_2)^{\theta}}{d(\overline{u_2},\overline{u_2}')^{\ee} N(g_2)^{\alpha}} + \frac{N(g_3)^{\theta} + N(g_3')^{\theta}}{\norm{g_3-g_3'}^{\ee} N(g_3)^{\beta} N(g_3')^{\beta}} \r)
\end{align*}
En passant à la limite quand $n \to +\infty$, on conclut que
\[
\norm{h}_{\scr B} \leq C.
\]

\textit{Point 2.} C'est une conséquence directe de \eqref{Estonie1} et du fait que les $(g_i)_{i \geq 1}$ sont identiquement distribués,
\[
\sup_{n\geq 0} \frac{\abs{\bf P_t^n h}_{\theta}}{\abs{h}_{\theta}} \leq \bb E \l( N(g_1)^{\theta} \r) < +\infty.
\]

\textit{Point 3.} Par le Lemme \ref{Biélorussie}, pour tout $t \in \bb R$, $n \geq 1$ et $h \in \scr B$
\[
\norm{\bf P_t^n h}_{\scr B} \leq c_{\ee} \l( 1+ \abs{t}^{\ee} \r) \abs{h}_{\theta} + c_{\ee} k_{\ee,\alpha}(h) r_{\ee}^n.
\]
Puisque $r_{\ee} \in (0,1)$, il existe $n_0 \geq 1$ tel que pour tout $n \geq n_0$, $c_{\ee} r_{\ee}^n \leq r < 1$ ce qui démontre le point 3.

\textit{Point 4.} Soit $B$ une partie bornée de $\scr B$. Montrons que $\bf P_t B$ est relativement compact. Soit $(h_n)_{n\geq 0}$ une suite à valeurs dans $B$, on va montrer qu'il existe une sous-suite de $\bf P_t h_n$ qui converge dans $\l( \scr C( \bb X, \bb C), \abs{\cdot}_{\theta} \r)$. Fixons $K$ un compact de $\bb X$. Pour tout $x = (g,\overline{u}) \in K$ et $n \geq 0$, par \eqref{Lettonie},
\begin{equation}
\label{Lituanie}
\abs{\bf P_t h_n (x)} \leq \bb E\l( N(g_1)^{\theta} \r) \abs{h_n}_{\theta}.
\end{equation}
Puisque $(h_n)_{n\geq 0}$ est bornée, on en déduit que $(\bf P_t h_n (x))_{n\geq 0}$ est bornée dans $\bb C$ et donc relativement compact. Montrons que $( \bf P_t h_n )_{n\geq 0}$ est équicontinue en $x \in K$. Pour tout $y = (g',\overline{v}) \in K$ et tout $n \geq 0$,
\begin{align*}
\abs{\bf P_t h_n(x) - \bf P_t h_n(y)} &\leq \abs{\bf P_t h_n(g,\overline{u}) - \bf P_t h_n(g,\overline{v})} + \abs{\bf P_t h_n(g,\overline{v}) - \bf P_t h_n(g',\overline{v})} \\
&\leq k_{\ee,\alpha}\l( \bf P_t h_n \r) d(\overline{u},\overline{v})^{\ee} N(g)^{\alpha} + k_{\ee,\beta}\l( \bf P_t h_n \r) \norm{g-g'}^{\ee} N(g)^{\beta} N(g')^{\beta}.
\end{align*}
Puisque $K$ est compact, il existe $c_K$ tel que $N(g) \leq c_K$ pour tout $g \in \bb G$. Donc en utilisant \eqref{Estonie2} et \eqref{Estonie3},
\begin{align*}
\abs{\bf P_t h_n(x) - \bf P_t h_n(y)} &\leq c_{\ee,K} \l( \abs{t}^{\ee} +1 \r) \norm{h_n}_{\scr B} \l(  d(\overline{u},\overline{v})^{\ee}  +  \norm{g-g'}^{\ee} \r).
\end{align*}
Puisque $(h_n)_{n\geq 0}$ est bornée, on en déduit que la suite $( \bf P_t h_n )_{n\geq 0}$ est équicontinue. Donc par le théorème d'Ascoli-Arzelà, l'ensemble $\{ \bf P_t h_n, n \geq 0 \}$ est relativement compact dans $( \scr C(K,\bb C), \abs{\cdot}_{\theta} )$. Par un procédé d'extraction diagonale, il existe une sous-suite $( \bf P_th_{n_k} )_{k\geq 0}$ et une fonction $\varphi \in \scr C( \bb X, \bb C)$ telles que pour tout compact $K$ de $\bb X$, on a
\[
\sup_{x \in K} \abs{\bf P_th_{n_k}(x) - \varphi(x)} \underset{k\to+\infty}{\longrightarrow} 0.
\]
De plus par \eqref{Lituanie}, pour tout $x \in \bb X$ et tout $n \geq 0$, $\abs{\bf P_t h_n (x)} \leq c_B$. Donc de même, pour tout $x \in \bb X$, $\varphi(x) \leq 1+c_B$. On en déduit que, pour tout $A > 0$ et tout $n \geq 0$,
\[
\sup_{x \in \bb X} \frac{ \abs{\bf P_t h_{n_k}(x) - \varphi(x)} }{N(g)^{\theta}} \leq \sup_{\substack{N(g) \leq A \\ \overline{u} \in \bb P(\bb R^d)}} \abs{\bf P_t h_{n_k}(g,\overline{u}) - \varphi(g,\overline{u})} +  \sup_{\substack{N(g) > A \\ \overline{u} \in \bb P(\bb R^d)}} \frac{ 1+2c_B }{N(g)^{\theta}}.
\]
Puisque $\{ (g,\overline{u}) \in \bb X : N(g) \leq A \}$ est compact, on a donc pour tout $A> 0$,
\[
\sup_{x \in \bb X} \frac{ \abs{\bf P_t h_{n_k}(x) - \varphi(x)} }{N(g)^{\theta}} \leq \frac{ 1+2c_B }{A^{\theta}}.
\]
En prenant la limite lorsque $A \to +\infty$, on conclut que la suite $( \bf P_t h_{n_k} )_{k\geq 0}$ converge dans $\l( \scr C( \bb X, \bb C), \abs{\cdot}_{\theta} \r)$ vers $\varphi$ et donc finalement $\bf P_t B$ est relativement compact dans $\l( \scr C( \bb X, \bb C), \abs{\cdot}_{\theta} \r)$.
\end{proof}

Avant de pouvoir établir M2,
nous avons besoin d'un autre résultat de Le Page \cite{page_theoremes_1982} (Corollaire 1) également énoncé dans Grama, Le Page et Peigné \cite{grama_conditioned_2016}. On rappelle que sous \ref{HypoP1}-\ref{HypoP3}, il existe une unique mesure $\bs \nu$ sur $\bb P(\bb R^d)$ qui soit $\bs \mu$-invariante, c'est-à-dire telle que pour tout fonction $\varphi$ : $\bb P(\bb R^d) \to \bb C$, continue,
\begin{equation}
\label{Malte}
(\bs \mu * \bs \nu)(\varphi) = \int_{\bb X} \varphi(g\cdot \overline{u}) \bs \nu(\dd \overline u) \bs \mu(\dd g) = \int_{\bb P(\bb R^d)} \varphi(\overline{u}) \bs \nu(\dd \overline{u}) = \bs \nu(\varphi).
\end{equation}

\begin{proposition}
\label{Moldavie}
Sous les conditions \ref{HypoP1}-\ref{HypoP3}, pour toute fonction $\varphi$ : $\bb P(\bb R^d) \to \bb C$ on a
\[
\lim_{n\to+\infty} \sup_{\overline{u}\in \bb P(\bb R^d)} \abs{\bb E\l( \varphi(G_n\cdot\overline{u}) \r) - \bs \nu(\varphi)} = 0.
\]
\end{proposition}

\begin{proposition}[Trou spectral]
\label{Canada2}
Supposons \ref{HypoP1}-\ref{HypoP3}. Alors
\begin{enumerate}[ref=\arabic*, leftmargin=*, label=\arabic*.]
	\item L'application $h \mapsto \bf Ph$ est un opérateur borné de $\scr B$.
	\item Il existe deux constantes $c_1 > 0$ et $c_2 > 0$ telle que
	\[
	\bf P = \Pi+Q,
	\]
	où $\Pi$ est un projecteur unidimensionnel et $Q$ est un opérateur de $\scr B$ tels que $\Pi Q=Q\Pi=0$. De plus pour tout $n\geq 1$,
	\[
	\norm{Q^n}_{\scr B\to\scr B} \leq c_1\e^{-c_2n}.
	\]
\end{enumerate}
\end{proposition}

\begin{proof}
\textit{Point 1.} C'est une simple conséquence du Lemme \ref{Biélorussie} pour $n=1$ et $t=0$.

\textit{Point 2.} D'après le Lemme \ref{Ukraine} et le théorème de Ionescu-Tulcea et Marinescu \cite{tulcea_theorie_1950}, on sait que qu'il existe un nombre fini de valeurs propres de module $1$ notées $\ll_1, \dots, \ll_p$ et des opérateurs $\Pi_1, \dots, \Pi_p,Q$ tels que $\bf P = \sum_{i=1}^p \ll_i \Pi_i +Q$ avec $\Pi_i$ des projecteurs unidimensionnels orthogonaux vérifiant $\Pi_i Q = Q \Pi_i = 0$ et le rayon spectral de $Q$ est strictement plus petit que $1$ et donc
\[
\norm{Q^n}_{\scr B\to\scr B} \leq c_1\e^{-c_2n}.
\]
Il ne nous reste plus qu'à montrer $1$ est l'unique valeur propre de module $1$ et qu'elle est simple. Soit $\ll \in \bb C$ une valeur propre de $\bf P$ de module $1$ et soit $h \in \scr B$ un vecteur propre associé. Pour tout $x=(g,\overline{u}) \in \bb X$ et tout $n \geq 1$,
\[
\ll^n h(x) = \bf P^n h (x) = \bb E \l( h( g_n, G_{n-1}g\cdot \overline{u} ) \r).
\]
Posons $\tt h(\overline{v}) = \bb E \l( h(g_1,\overline{v}) \r)$, pour tout $\overline{v} \in \bb P(\bb R^d)$. Par indépendance des $g_i$,
\[
\ll^n h(x) = \bb E \l( \tt h( G_{n-1}g\cdot \overline{u} ) \r)
\]
Or puisque $h \in \scr B$, $\tt h$ est $\ee$-hölderienne : pour tout $\overline{v}$ et $\overline{w} \in \bb P(\bb R^d)$,
\[
\abs{\tt h(\overline{v}) - \tt h(\overline{w})} \leq k_{\ee,\alpha}(h) \bb E \l( N(g_1)^{\alpha} \r) d(\overline{v},\overline{w})^{\ee}.
\]
La fonction $\tt h$ est notamment continue donc par la Proposition \ref{Moldavie},
\[
\ll^n h(x) \underset{n\to +\infty}{\longrightarrow} \bs \nu(\tt h) = \int_{\bb X} h(g_1,\overline{v}) \bs \nu(\dd \overline{v}) \bs \mu( \dd g_1).
\]
Or puisque $h$ est un vecteur propre, par définition, il existe $x_0 \in \bb X$ tel que $h(x_0) \neq 0$ donc $\abs{\ll^n h(x_0)} = \abs{h(x_0)} = \bs \nu(\tt h)$. Par conséquent, $\ll = 1$ et pour tout $x \in \bb X$, $h(x) = \bs \nu(\tt h) e$ est un multiple de la fonction constante égale à $1$. Ceci montre que $1$ est l'unique valeur propre de module $1$ et son sous-espace vectoriel propre et de dimension $1$, ce qui conclut la preuve du point 2.
\end{proof}

\begin{proposition}[Opérateur perturbé]
\label{Canada3}
Supposons \ref{HypoP1}-\ref{HypoP3}. Alors
\begin{enumerate}[ref=\arabic*, leftmargin=*, label= \arabic*.]
	\item Pour tout $\abs{t} \leq \kappa$ l'application $g \mapsto \bf P_tg$ est un opérateur borné sur $\scr B$.
	\item Il existe une constante $C_{\bf P} >0$ telle que, pour tout $n\geq 1$ et $\abs{t} \leq \kappa$,
	\[
	\norm{\bf P_t^n}_{\scr B \to \scr B} \leq C_{\bf P}.
	\]
\end{enumerate}
\end{proposition}

\begin{proof}
C'est une conséquence immédiate du Lemme \ref{Biélorussie}.
\end{proof}

\subsection{Preuve de M4}

On pose pour tout $h \in \scr B$,
\begin{equation}
\label{Chypre}
\tbs \nu (h) = \int_{\bb X} h(g_1,\overline{v}) \bs \nu(\dd \overline{v}) \bs \mu( \dd g_1).
\end{equation}

\begin{proposition}[Intégrabilité locale]
\label{Canada4}
Supposons \ref{HypoP1}.
L'espace de Banach $\scr B$ contient une suite de fonctions positives $\tt N, \tt N_1, \tt N_2, \dots $ telle que :
\begin{enumerate}[ref=\arabic*, leftmargin=*, label=\arabic*.]
\item Il existe $p_{max} > 2$ et $\gamma > 0$ telles que pour tout $x\in \bb X$,
	\[
	\max \left\{ \abs{\rho(x)}^{1+\gamma}, \norm{\bs \delta_x}_{\scr B'}, \bb E_x^{1/p_{max}} \left( \tt N\l( X_n \r)^{p_{max}} \right) 
	\right\} \leq c \left( 1+\tt N(x) \right)
	\]
et
	\[
	\tt N(x) \bbm 1_{\{ \tt N(x) > l\}} \leq \tt N_l(x), \quad \text{pour tout} \quad l\geq 1.
	\]
\item Il existe $c > 0$ telle que pour tout $l\geq 1$,
	\[
	\norm{\tt N_l}_{\scr B} \leq  c.
	\]
\item Il existe $\delta>0$ et $c > 0$ telles que pour tout $l\geq 1$,
	\[
	 \abs{\tbs \nu \left( \tt N_l \right)} \leq \frac{c}{l^{1+\delta}}.
	\]
\end{enumerate}
\end{proposition}

\begin{proof}
Pour tout $x=(g,\overline{u}) \in \bb X$ on pose $\tt N(x) = \l( \norm{g} + \norm{g^{-1}} \r)^{\theta}$. Pour tout $t \geq 0$ et tout $l \geq 1$, on considère
\[
\phi_l(t) = \left\{ \begin{array}{ccl} 0 & \text{si} & t \leq l-1 \\ t-(l-1) & \text{si} & t \in [l-1,l] \\ 1 & \text{si} & t \geq l \end{array} \right.
\]
et on pose pour tout $x \in \bb X$ et $l \geq 1$, $\tt N_l(x) = \phi_l\l( \tt N(x) \r) \tt N(x)$.

Montrons que $\tt N \in \scr B$. Il est facile de voir que $\abs{\tt N}_{\theta} \leq 2^{\theta}$, que $k_{\ee,\alpha}(\tt N) = 0$. De plus pour tout $(g,g') \in \bb G$ et $\overline{u} \in \bb P(\bb R^d)$, par le théorème des accroissements finis,
\[
\abs{\tt N(g,\overline{u}) - \tt N(g',\overline{u})} \leq \theta \sup_{\xi \geq 1} \xi^{\theta-1} \abs{ \norm{g} + \norm{g^{-1}} - \norm{g'} - \norm{(g')^{-1}} }
\]
Sans perte de généralité, on peut supposer que $\delta_0 \leq 8/3$ et donc
\[
\abs{\tt N(g,\overline{u}) - \tt N(g',\overline{u})} \leq \theta \l( \norm{g-g'} + \norm{g^{-1}(g-g')(g')^{-1}} \r) \leq 2\theta \norm{g-g'} N(g)N(g').
\]
D'autre part,
\[
\abs{\tt N(g,\overline{u}) - \tt N(g',\overline{u})} \leq 2^{\theta} N(g)^{\theta} + 2^{\theta} N(g')^{\theta} \leq 2^{\theta+1} N(g)^{\theta} N(g')^{\theta}.
\]
Par conséquent,
\begin{equation}
\label{Roumanie}
\abs{\tt N(g,\overline{u}) - \tt N(g',\overline{u})} \leq 2^{(\theta+1)(1-\ee)+\ee} \norm{g-g'}^{\ee} N(g)^{\ee + (1-\ee)\theta} N(g')^{\ee + (1-\ee)\theta}.
\end{equation}
Or $\ee + (1-\ee)\theta \leq 4\ee \leq \beta$ et donc
\[
k_{\ee,\beta}'\l( \tt N \r) \leq 2^{\theta+1}.
\]
Ainsi $\tt N \in \scr B$ et $\norm{\tt N}_{\scr B} \leq 2^{\theta+2}$.

Montrons maintenant que pour tout $l \geq 1$, $\tt N_l \in \scr B$. Soit $l \geq 1$. On remarque facilement que pour tout $x \in \bb X$, $\tt N_l(x) \leq \tt N(x)$ et donc $\abs{\tt N_l}_{\theta} \leq \abs{\tt N}_{\theta} \leq 2^{\theta}$. Puisque pour tout $g \in \bb G$, la fonction $\overline{u} \mapsto \tt N(g,\overline{u})$ est constante, il en est de même pour $\overline{u} \mapsto \tt N_l(g,\overline{u})$ et donc $k_{\ee,\alpha}(\tt N_l) = 0$. Pour tout $(g,g') \in \bb G$ et $\overline{u} \in \bb P(\bb R^d)$,
\begin{align*}
\abs{\tt N_l(g,\overline{u}) - \tt N_l(g',\overline{u})} &\leq \tt N(g,\overline{u}) \abs{\phi_l\l( \tt N(g,\overline{u}) \r) - \phi_l\l( \tt N(g',\overline{u}) \r)} + \norm{\phi_l}_{\infty} \abs{\tt N(g,\overline{u}) - \tt N(g',\overline{u})} \\
&\leq \l( 2^{\theta} N(g)^{\theta} \norm{\phi_l'}_{\infty} + 1 \r) \abs{\tt N(g,\overline{u}) - \tt N(g',\overline{u})}.
\end{align*}
En utilisant \eqref{Roumanie},
\begin{align*}
\abs{\tt N_l(g,\overline{u}) - \tt N_l(g',\overline{u})} &\leq \l( 2^{\theta} + 1 \r) N(g)^{\theta} 2^{\theta+1} \norm{g-g'}^{\ee} N(g)^{\ee + \theta} N(g')^{\ee + \theta} \\
&\leq 2^{2\theta+2} \norm{g-g'}^{\ee} N(g)^{\ee + 2\theta} N(g')^{\ee + 2\theta}.
\end{align*}
Finalement, puisque $\ee + 2\theta = 7\ee = \beta$, on conclut que $k_{\ee,\beta}'(N_l) \leq 2^{2\theta+2}$, que $N_l \in \scr B$ et que
\begin{equation}
\label{Grèce}
\norm{N_l}_{\scr B} \leq 2^{2\theta+3}.
\end{equation}

\textit{Point 1.} On rappelle que pour tout $x=(g,\overline{u}) \in \bb X$, on a $\rho(x) = \log \l( \norm{gu} \r)$. Soit $\gamma > 0$ (par exemple $\gamma =1$). Si $\norm{gu} \geq 1$, alors $\abs{\rho(x)}^{1+\gamma} \leq \log \l( \norm{g} \r)^{1+\gamma} \leq c_{\gamma,\theta} \norm{g}^{\theta}$. Si $\norm{gu} \leq 1$, alors $\abs{\rho(x)}^{1+\gamma} \leq -\log \l( \norm{g^{-1}}^{-1} \r)^{1+\gamma} \leq c_{\gamma,\theta} \norm{g^{-1}}^{\theta}$. Dans tous les cas,
\[
\abs{\rho(x)}^{1+\gamma} \leq c_{\gamma,\theta} \tt N(x).
\]
Dans \eqref{Bulgarie}, on a vu que $\norm{\bs \delta_x}_{\scr B'} \leq N(g)^{\theta} \leq \tt N(x)$, pour tout $x=(g,\overline{u}) \in \bb X$. On choisit $p_{max} = 8/3 > 2$. On a alors $\theta p_{\max} = \delta_0$ et donc
\[
\bb E_x^{1/p_{max}} \l( \tt N(X_n)^{p_{\max}} \r) \leq 2^{\theta} \bb E^{1/p_{max}} \l( N(g_n)^{\theta p_{\max}} \r) = c_{\delta_0} <+\infty.
\]
Enfin, par définition de $\phi_l$, il est clair que pour tout $l \geq 1$ et tout $x \in \bb X$, $\tt N(x) \bbm 1_{\{ \tt N(x) > l\}} \leq \tt N_l(x)$, ce qui achève de démontrer le point 1.

\textit{Point 2.} Ce point a été démontré en \eqref{Grèce}.

\textit{Point 3.} Prenons $\delta = 2/3$. Par définition de $\tt N_l$, on a pour tout $l \geq 2$,
\begin{align*}
\tbs \nu\l( \tt N_l \r) = \int_{\bb X} \tt N_l(g_1,\overline{v}) \bs \nu(\dd \overline{v}) \bs \mu( \dd g_1) &\leq \int_{\bb X}  \tt N(g_1,\overline{v}) \bbm 1_{\{ \tt N(g_1,\overline{v}) \geq l-1 \}} \bs \nu(\dd \overline{v}) \bs \mu( \dd g_1) \\
&\leq \int_{\bb X} \frac{\tt N(g_1,\overline{v})^{2+\delta}}{(l-1)^{1+\delta}} \bs \nu(\dd \overline{v}) \bs \mu( \dd g_1) \\
&\leq \frac{2^{\theta(2+\delta)}}{(l-1)^{1+\delta}} \int_{\bb G} N(g_1)^{\theta(2+\delta)} \bs \mu( \dd g_1)
\end{align*}
Comme $\theta(2+\delta) = \delta_0$, on conclut que pour tout $l \geq 1$, $\tbs \nu\l( \tt N_l \r) \leq c_{\delta_0}/l^{5/3}$.
\end{proof}

Soient $\mu$ et $\sigma^2$ définis respectivement par le point 1 et 2 de la Proposition 2.1. 
La mesure $\tbs \nu$ définie par \eqref{Chypre} est $\bf P$-invariante. En effet pour toute fonction $h$ : $\bb X \to \bb C$ continue et bornée,
\begin{align*}
\label{Albanie}
\tbs \nu\l( \bf Ph \r) = \int_{\bb X} \bf P h(g,\overline{u}) \bs \nu(\dd \overline{u}) \bs \mu( \dd g) &= \int_{\bb X\times \bb G} h(g_1,g\cdot\overline{u}) \bs \mu( \dd g_1)\bs \mu( \dd g)\bs \nu(\dd \overline{u}) \\
&= \int_{\bb X} \tt h(g\cdot\overline{u}) \bs \mu( \dd g)\bs \nu(\dd \overline{u}).
\end{align*}
où $\tt h$ est définie par $\tt h(\overline{v}) = \int_{\bb G} h(g_1,\overline{v}) \bs \mu( \dd g_1)$, pour tout $\overline{v} \in \bb P(\bb R^d)$. Puisque $\bs \nu$ est $\mu$-invariante (cf \eqref{Malte}),
\[
\tbs \nu\l( \bf Ph \r) = \int_{\bb P(\bb R^d)} \tt h(\overline{u}) \bs \nu(\dd \overline{u}) = \tbs \nu (h).
\]
Notons également que $\tbs \nu(\tt N^2) \leq 2^{2\theta} \bf \mu (N^{2\theta})$ et comme $2\theta \leq \delta_0$ on en déduit que $\tbs \nu(\tt N^2) <+\infty$. Ainsi l'équation (2.5) 
est satisfaite.

\begin{proposition}[Centrage et non-dégénérescence]
\label{Canada5}
Supposons \ref{HypoP1}-\ref{HypoP4}.
La marche $( S_n )_{n\geq 1}$ est centrée :
\[
\tbs \nu (\rho) = \mu = 0,
\]
et non-dégénérée :
\[
\sigma^2 = \Var_{\tbs \nu} \l( \rho(X_1) \r) + 2\sum_{k=2}^{+\infty} \Cov_{\tbs \nu} \l( \rho(X_1), \rho(X_k) \r) > 0.
\]
\end{proposition}

\begin{proof}
L'hypothèse \ref{HypoP4} nous dit exactement que $\tbs \nu (\rho) = 0$. En ce qui concerne la non-dégénérescence, par le théorème 2 de Le Page \cite{page_theoremes_1982}, sous les hypothèses \ref{HypoP1}-\ref{HypoP4}, on sait que $\frac{1}{n} \bb E_x \l( S_n^2 \r)$ converge vers un réel strictement positif noté $\tt \sigma^2 > 0$. Donc par le point 2 de la Proposition 2.1, 
$\sigma^2 = \tt \sigma^2 > 0$.
\end{proof}

\subsection{Formulation des résultats}

A travers les Propositions \ref{Canada1}, \ref{Canada2}, \ref{Canada3} et \ref{Canada4}, nous avons montré que les Hypothèses M1-M5 
était vérifiées et donc les Théorèmes 2.2-2.5 
précisent le comportement de la marche aléatoire $( S_n )_{n\geq 1}$ associée.

\begin{proposition}
Supposons \ref{HypoP1}-\ref{HypoP4}. Alors les résultats des Théorèmes 2.2-2.5 
s'appliquent à la marche aléatoire issue du produit de matrices $S_n = \log \l( \norm{G_ng \cdot x} \r)$.
\end{proposition}

Détaillons les quelques points faisant appel à la fonction $\tt N$. Pour tout $\gamma > 0$, on définit
\[
\scr D_{\gamma}' := \left\{ (x,y) \in \bb X \times \bb R : \exists n_0 \geq 1, \; \bb P_x \l( y+S_{n_0} > \gamma \,,\, \tau_y > n_0 \r) > 0 \right\}
\]

\begin{proposition}
\label{Macédoine1}
Supposons \ref{HypoP1}-\ref{HypoP4}. 
\begin{enumerate}[ref=\arabic*, leftmargin=*, label=\arabic*.]
	\item La fonction $V$ du Théorème 2.2 
	vérifie, pour tout $y \in \bb R$, $x \in \bb X$ et $\delta > 0$,
	\[
	(1-\delta) \max(y,0) - c_{\delta} \leq V(x,y) \leq (1+\delta) \max(y,0) + c_{\delta}.
	\]
	\item Il existe $\gamma_0' > 0$ tel que pour tout $\gamma > \gamma_0'$,
	\[
	\supp(V) = \scr D_{\gamma}'.
	\]
\end{enumerate}
\end{proposition}

\begin{proof}
\textit{Point 1.} Par le point 2 du Théorème 2.2, 
pour tout $(x,y) \in \bb X \times \bb R$, on a $V(x,y) = \bb E_x \l( V(X_1,y+S_1) \,;\, \tau_y > 1 \r)$. Donc en utilisant le point 3 du Théorème 2.2, 
\begin{align*}
V(x,y) &\leq \bb E_x \l( (1+\delta)(y+S_1) + c_{\delta}(1+\tt N(X_1)) \,;\, \tau_y > 1 \r) \\
&\leq (1+\delta) \max(y,0) + c_{\delta} \bb E_x \l( 1+\abs{\rho(X_1)} + \tt N(X_1) \r).
\end{align*}
Par le point 1 de la Proposition \ref{Canada4},
\[
V(x,y) \leq (1+\delta) \max(y,0) + c_{\delta,\theta} \bb E \l( N(g_1)^{\theta} \r) \leq (1+\delta) \max(y,0) + c_{\delta,\theta}.
\]
On procède de façon similaire pour la minoration :
\begin{align*}
V(x,y) &\geq (1-\delta) \bb E_x \l( y+S_1 \,;\, y+S_1 > 0 \r) - c_{\delta} \bb E_x \l( 1+\tt N(X_1) \,;\, \tau_y > 1 \r) \\
&\geq (1-\delta) y + (1-\delta)\bb E_x \l( S_1 \r) - (1-\delta) \bb E_x \l( y+S_1 \,;\, y+S_1 \leq 0 \r) - c_{\delta,\theta} \\
&\geq (1-\delta) y + (1-\delta)\bb E_x \l( S_1 \r) - c_{\delta,\theta}.
\end{align*}
Puisque $V \geq 0$, on conclut que
\[
V(x,y) \geq (1-\delta) \max(y,0) - c_{\delta,\theta}
\]

\textit{Point 2.} En prenant $\delta = 1/2$ dans le point 1, il existe $\gamma_0' = 4c_{\delta} > 0$ telle que pour tout $(x,y) \in \bb X \times \bb R$,
\[
V(x,y) \geq \frac{y}{2} - \frac{\gamma_0'}{4}.
\]
On rappelle que, les ensembles $\scr D_{\gamma}$ sont définis avant le Théorème 2.2 
par
\[
\scr D_{\gamma} := \left\{ (x,y) \in \bb X \times \bb R : \exists n_0 \geq 1,\; \bb P_x \l( y+S_{n_0} > \gamma \l( 1+\tt N \left( X_{n_0} \r) \r) \,,\, \tau_y > n_0 \r) > 0 \right\}
\]
et donc puisque $\tt N \geq 0$, pour tout $\gamma > 0$, $\scr D_{\gamma} \subseteq \scr D_{\gamma}'$. Ainsi par le point 3 de la Proposition 8.8,  
on en déduit que pour tout $\gamma > 0$,
\begin{equation}
\label{Serbie}
\supp(V) \subseteq \scr D_{\gamma}'.
\end{equation}

D'autre part, pour tout $\gamma > 0$, posons $\zeta_{\gamma}' := \inf \left\{ k \geq 1 : \abs{y+S_k} > \gamma \right\}$ et considérons $(x,y) \in \scr D_{\gamma_0}'$. Par définition, il existe $n_0\geq 1$ tel que $\bb P_x \l( y+S_{n_0} > \gamma \,,\, \tau_y > n_0 \r)$. Alors, comme dans la démonstration du point 4 de la Proposition 8.8, 
\begin{align*}
V(x,y) &\geq \bb E_x \l( V(X_{n_0},y+S_{n_0}) \,;\, \tau_y > n_0 \,,\, \zeta_{\gamma_0}' \leq n_0 \r) \\
&\geq \frac{1}{2}\bb E_x \l( y+S_{\zeta_{\gamma_0}'} - \frac{\gamma_0'}{2} \,;\, \tau_y > \zeta_{\gamma_0}' \,,\, \zeta_{\gamma_0}' \leq n_0 \r) \\
&\geq \frac{\gamma_0'}{4} \bb P_x \l( \tau_y > \zeta_{\gamma_0}' \,,\, \zeta_{\gamma_0}' \leq n_0 \r) > 0.
\end{align*}
Donc $\scr D_{\gamma_0}' \subseteq \supp(V)$. Or il est facile de voir que que pour tout $\gamma_1 \geq \gamma_2$, $\scr D_{\gamma_1}' \subseteq \scr D_{\gamma_2}'$. Ainsi pour tout $\gamma \geq \gamma_0$, $\scr D_{\gamma}' \subseteq \supp(V)$ ce qui, avec \eqref{Serbie}, conclut la preuve.
\end{proof}

\begin{proposition}
\label{Macédoine2}
Supposons \ref{HypoP1}-\ref{HypoP4}. Alors pour tout $(x,y) \notin \supp(V)$,
\[
\bb P_x \l( \tau_y > n \r) \leq c\e^{-cn}.
\]
\end{proposition}

\begin{proof} Soit $(x,y) \notin \supp(V)$. Par le point 2 du Théorème 2.2, 
$0 = V(x,y) = \bb E_x \l( V(X_1,y+S_1) \,;\, \tau_y > 1 \r)$. Donc sur l'évènement $\{ \tau_y > 1 \}$, on a $(X_1,y+S_1) \notin \supp (V)$. Par conséquent, en utilisant la propriété de Markov et le point 2 du Théorème 2.3,
pour tout $n \geq 2$,
\begin{align*}
\bb P_x \l( \tau_y > n \r) \leq c \e^{-c(n-1)} \bb E_x \l( 1 + \tt N(X_1) \,;\, \tau_y > 1 \r) \leq c \e^{-cn} \bb E \l( N(g_1)^{\theta} \r).
\end{align*}
\end{proof}

Avec les mêmes idées que celles présentées dans les preuves des Propositions \ref{Macédoine1} et \ref{Macédoine2} et en utilisant les Théorèmes 2.4 et 2.5, 
on obtient la proposition suivante.

\begin{proposition}
\label{Macédoine3}
Supposons \ref{HypoP1}-\ref{HypoP4}.
\begin{enumerate}[ref=\arabic*, leftmargin=*, label=\arabic*.]
\item Il existe $\ee_0' > 0$ tel que pour tout $\ee' \in(0,\ee_0')$, $n\geq 1$ et $(x,y) \notin \supp(V)$,
	\[
	\abs{\bb P_x \left( \tau_y > n \right) - \frac{2V(x,y)}{\sqrt{2\pi n} \sigma}} \leq c_{\ee'}\frac{\l( 1+\max(y,0) \r)^2}{n^{1/2+\ee'}}.
	\]
	\item Pour tout $(x,y) \in \bb X \times \bb R$ et $n\geq 1$,
	\[
	\bb P_x \left( \tau_y > n \right) \leq c\frac{ 1 + \max(y,0) }{\sqrt{n}}.
	\]
	\item Il existe $\ee_0' >0$ tel que pour tout $\ee' \in (0,\ee_0')$, $n\geq 1$, $t_0 > 0$, $t\in [ 0, t_0 ]$ et $(x,y)\in \bb X \times \bb R$,
\[
\abs{\bb P_x \left( y+S_n \leq t \sqrt{n} \,,\, \tau_y > n \right) -\frac{2V(x,y)}{\sqrt{2\pi n}\sigma} \mathbf \Phi^+\left(\frac{t}{\sigma}\right)} \leq c_{\ee',t_0}\frac{\l( 1+\max(y,0) \r)^2}{n^{1/2+\ee'}}.
\]
\end{enumerate}
\end{proposition}

\bibliographystyle{plain}
\bibliography{biblio7}

\begin{thebibliography}{1}

\bibitem{grama_limit_2018-1}
I.~Grama, R.~Lauvergnat, and E.~Le Page.
\newblock Limit theorems for {Markov} walks conditioned to stay positive under
  a spectral gap assumption.
\newblock {\em The Annals of Probability}, pages 1807--1877, 2018.

\bibitem{grama_conditioned_2016}
I.~Grama, \'E.~Le Page, and M.~Peign\'e.
\newblock Conditioned limit theorems for products of random matrices.
\newblock {\em Probability Theory and Related Fields}, pages 1--39, 2016.

\bibitem{tulcea_theorie_1950}
C.~T. Ionescu~Tulcea and G.~Marinescu.
\newblock Th\'eorie ergodique pour des classes d'op\'erations non
  compl\`etement continues.
\newblock {\em Annals of Mathematics}, 52(1):140--147, 1950.

\bibitem{page_theoremes_1982}
E.~Le~Page.
\newblock Th\'eor\`emes limites pour les produits de matrices al\'eatoires.
\newblock In {\em Probability {Measures} on {Groups}}, pages 258--303.
  Springer, Berlin, Heidelberg, 1982.
\newblock DOI: 10.1007/BFb0093229.

\bibitem{norman1972markov}
M.~F. Norman.
\newblock {\em Markov processes and learning models}, volume~84.
\newblock 1972.

\end{thebibliography}

\end{document}